\documentclass[a4paper,12pt]{amsart}
\usepackage{amsmath,amsthm,amssymb}
\usepackage{hyperref}
\usepackage{cite}

\allowdisplaybreaks[1]

\newcommand{\ti}{\widetilde}
\newcommand{\la}{\lambda}
\newcommand{\ga}{\gamma}
\newcommand{\ity}{\infty}
\newcommand{\C}{\mathbb{C}}
\newcommand{\R}{\mathbb{R}}
\newcommand{\N}{\mathbb{N}}

\newcommand{\F}{\mathfrak{F}}
\newcommand{\G}{\mathfrak{G}}

\numberwithin{equation}{section}
\newtheorem{theorem}{Theorem}[section]
\newtheorem{lemma}[theorem]{Lemma}
\newtheorem{corollary}[theorem]{Corollary}
\theoremstyle{remark}
\newtheorem{remark}[theorem]{Remark}
\newtheorem{example}[theorem]{Example}
\newtheorem{definition}[theorem]{Definition}

\thanks {The research work of the first  author is supported by research fellowship from Council of Scientific and Industrial Research (CSIR), New Delhi.}

\begin{document}

\title[Semigroups of transcendental functions]{the dynamics of semigroups of transcendental entire functions I}
\author[D. Kumar]{Dinesh Kumar}
\address{Department of Mathematics, University of Delhi,
Delhi--110 007, India}

\email{dinukumar680@gmail.com }

\author[S. Kumar]{Sanjay Kumar}

\address{Department of Mathematics, Deen Dayal Upadhyaya College, University of Delhi,
Delhi--110 007, India }

\email{sanjpant@gmail.com}

\begin{abstract}
We consider the dynamics associated with an arbitrary semigroup of transcendental entire functions. Fatou-Julia theory is used to investigate the dynamics of these semigroups. Several results of the dynamics associated with iteration of a transcendental entire function have been extended to  transcendental semigroup case. We also investigate the dynamics of conjugate semigroups, abelian transcendental semigroups and wandering and Baker domains of transcendental semigroups.
\end{abstract}


\maketitle

\section{Introduction}\label{sec1}
A natural extension of the dynamics associated to the iteration of a complex function is the dynamics of composite of two or more such functions and this leads to the realm of semigroups of transcendental entire functions. In this direction the seminal work was done by Hinkkanen and Martin \cite{martin} related to semigroups of rational functions. In their  paper, Hinkkanen and Martin extended the classical theory of the dynamics  associated to the iteration of a rational function of one complex variable to the more general setting of an arbitrary semigroup of rational functions. Many of the results were extended to semigroup of transcendental entire functions by Poon \cite{poon1}, Zhigang \cite{zhigang} and Huang and Cheng \cite{cheng}.
A transcendental semigroup $G$ is a semigroup generated by a family of transcendental entire function $\{f_1,f_2,\ldots\}$ with the semigroup operation being functional composition. We denote the semigroup by $G=[f_1,f_2,\ldots].$ Thus each $g\in G$ is a transcendental entire function and $G$ is closed under composition. For an introduction to iteration theory of entire functions, see \cite {berg1}.\\

A family $\F$ of meromorphic functions is normal in a domain $D\subset\C$ if every sequence in $\F$ has a subsequence which converges locally uniformly in $D$ to a meromorphic function or to the constant $\ity.$ 
The set of normality or the Fatou set $F(G)$ of a transcendental semigroup $G$, is the largest open subset of $\C$ on which the family of functions in $G$ is normal. Thus $z_0\in F(G)$ if it has a neighborhood $U$ on which the family $\{g:g\in G\}$ is  normal. The Julia set $J(G)$ of $G$ is the complement of $F(G),$ that is $J(G)=\ti\C\setminus F(G).$ The semigroup generated by a single function $g$ is denoted by $[g].$ In this case we denote $F([f])$ by $F(f)$ and $J([f])$ by $J(f)$ which are the respective Fatou set and Julia set in the classical Fatou-Julia theory of iteration of a single transcendental entire function. The dynamics of a semigroup is more complicated than those of a single function. Some of the properties in the classical dynamics do not get preserved in the semigroup case. For instance $F(G)$ and $J(G)$ need not be completely invariant and $J(G)$ may not be the entire complex plane $\C$ even if $J(G)$ has an interior point, see \cite{martin}. In this paper we have generalised  the dynamics of a transcendental entire function on its Fatou set to the dynamics of semigroup of transcendental entire functions. Furthermore the dynamics of conjugate semigroups, abelian transcendental semigroups and wandering and Baker domains of transcendental semigroups  have also been investigated. 
\section{Results from dynamics of one transcendental entire function}\label{sec2}

%
%
In this section we are listing the results of complex dynamics of a transcendental entire function which we wish to extend in the context of transcendental semigroup.  For $n\in\N$ let $f^n$ denote the n-th iterate of $f.$


\begin{theorem}\cite[Theorem 3.1]{baker2}\label{sec2,thm3}
Let $f$ be a transcendental entire function and $U$ be a multiply connected component of  $F(f)$.
Then
\begin{enumerate}
\item\ $f^n\to\ity$ locally uniformly on $U;$ 
\item\ Suppose $\ga$ is a Jordan curve that is not contractible in $U.$ Then distance between $f^n(\gamma)$ and $0$ is large for all sufficiently large $n$. Also $ind_0\,f^n(\ga)>0$ for all sufficiently large $n$ and $ind_0\,f^n(\ga)\to\ity$ as $n\to\ity,$ (where $ind_{\zeta}\,\ga$ denotes the index of a curve $\ga\subset\C$ with respect to a point $\zeta$).
\end{enumerate}
\end{theorem}



\begin{theorem}\cite[Corollary]{baker2}\label{sec2,thm4}
Let $f$ be a transcendental entire function which is bounded on some curve $\Gamma$ going to $\ity.$ Then all the components of $F(f)$ are simply connected.
\end{theorem}


\begin{theorem}\cite[Theorem 1]{baker1}\label{sec2,thm5}
Let $f$ be a transcendental entire function.Then every unbounded component  of $F(f)$ is simply connected.
\end{theorem}


\begin{theorem}\cite[Corollary]{baker1}\label{sec2,thm6}
Let $f$ be a transcendental entire function. If $F(f)$ has an unbounded component $U,$ then all  components of $F(f)$ are simply connected.
\end{theorem}





\begin{theorem}\cite[Lemma 6]{baker3}\label{sec2,thm9}
For any transcendental entire function $f,$ any doubly connected component of $F(f)$ does not contain critical points of $f.$
\end{theorem}
As an application of Picard's theorem and Theorem \ref{sec2,thm6}  one obtains the following result:
\begin{theorem}\cite[p.\ 278]{baker1}\label{sec2,thmf}
Let $f$ be a transcendental entire function and $U\subset F(f)$ be a completely invariant domain. Then
\begin{enumerate}
\item\ $U$ is unbounded;
\item\ all components of $F(f)$ are simply connected.
\end{enumerate}
\end{theorem}
Now we state an extensively cited result  of Baker concerning multiply connected domains of normality:
\begin{theorem}\cite{baker2}\label{sec2,thm10}
Let $f$ be a transcendental entire function. Then a multiply connected component of $F(f)$ is bounded and wandering, and hence a pre-periodic component of $F(f)$ must be simply connected.
\end{theorem}
 From the classification of periodic components of $F(g)$ of a  transcendental meromorphic function $g,$ see\cite[Theorem 2.2\,(iv)]{baker4} we have: If $D\subset F(g)$ is a periodic component of period say $m,$ then $D$ is a Herman ring if it is doubly connected and there exist an analytic homeomorphism $\psi: D\to A,$ where $A$ is an annulus $A=\{z:1<|z|<r\}, r>1,$ such that $\psi\circ g^m\circ\psi^{-1}(z)=e^{2\,\pi\,i\,\alpha}z$ for some $\alpha\in\R\setminus\mathbb Q.$
As an application of Theorem \ref{sec2,thm10}, one obtains the following result:
\begin{theorem}\cite[p.\ 65]{Hua}\label{sec2,thm8}
If $f$ is a transcendental entire function, then $F(f)$ does not contain Herman rings.
\end{theorem}
The proofs of these results can be found in Baker \cite{baker1,baker2,baker3}, and  Hua and Yang \cite{Hua}. In \cite{dinesh} using approximation theory of entire functions, the authors have shown the existence of entire functions $f$ and $g$ having infinite number of domains satisfying various properties and relating it to their composition. They explored and enlarged all the maximum possible ways of the solution in comparison to the past result worked out. It would be interesting to explore such relations for transcendental semigroup $G$ and its constituent elements.
%
\section{Classification of periodic components and definitions}\label{sec5}
This section contains the classification of periodic components of $F(G)$ for a transcendental semigroup $G$ in  analogy to the work of Hinkkanen and Martin, \cite{martin} on the classification of periodic components of $F(G')$ for a rational semigroup $G'$.
We give the  classification of the dynamics of a transcendental semigroup $G$   on an invariant component $V$ of $F(G).$\\
 Let $G$ be a transcendental semigroup and let $U$ be a component of the Fatou set $F(G)$ of $G$. Denote by $U_g$  the component of $F(G)$ containing $g(U)$.  Define the \emph{stabilizer} of $U$ to be $G_U=\{g\in G:U_g=U\}$. If $G_U$ is non-empty we say $U$ is  a stable basin for $G$. Given a stable basin $U$  for $G$, it is

\begin{enumerate}
\item[(i)] attracting if $U$ is a subdomain of an attracting basin of each $g\in G_U;$
\item[(ii)] superattracting if $U$ is a subdomain of a superattracting basin of each $g\in G_U;$
\item[(iii)] parabolic if $U$ is a subdomain of a parabolic basin of each $g\in G_U;$
\item[(iv)] Siegel   if $U$ is a subdomain of a Siegel disk of each $g\in G_U;$
\item[(v)] Baker if $U$ is a subdomain of a Baker domain of each $g\in G_U.$
\end{enumerate}

It should be noted that in the classical case a stable basin is one of the above types. We now propose a definition:
\begin{definition}\label{sec5,defn1}
 If $G$ is a transcendental semigroup and $U$ is a multiply connected component of $F(G),$ define 
\[
\ti\G_U=\{g\in G: F(g)\,\text{has a multiply connected component}\, \ti U_g\supset U\}.
\]
Observe that $\ti\G_U$ is non empty (see the argument in the proof of Theorem \ref{sec3,thm11}).
 $\ti\G_U$ need not be a subsemigroup of $G$ and so we consider the subsemigroup generated by $\ti\G_U$ and throughout the paper denote it by $\ti\G_U$ itself. Then  $F(G)\subset F(\ti\G_U)$; in general equality may not hold, and it is interesting to know if $U$ can be a proper subset of a component of $F(\ti\G_U).$
\end{definition}
We next give some examples   of  transcendental semigroup $G$ whose Fatou set $F(G)$ is non empty:
\begin{example}\label{sec5,eg1}
Let $G=[e^z+\la, e^z+\la+2\pi i],$ where $\la\leq -1.$ It can be seen that $F(e^z+\la)=F(e^z+\la+2\pi i)$ as $(e^z+\la+2\pi i)^n=(e^z+\la)^n+2\pi i$ for each $n\in\N.$ Let $g=e^z+\la.$ Observe that for $p, q, r, s \in\N,\, g^p\circ (g+2\pi i)^q=g^{p+q}$ and $(g+2\pi i)^r\circ g^s=g^{r+s}+2\pi i.$ Then for any $f\in G,$ either $f$ equals $g^k,$ for some $k\in\N,$ or  $f=g^m+2\pi i,$ for some $m\in\N.$ In either of the cases, $F(g)=F(f)$ and hence $F(G)=F(e^z+\la)\neq\emptyset,$ \cite{R4}. 
\end{example}
\begin{example}\label{sec5,eg2}
Let $G=[\la\sin{z},\la\sin z+2\pi]$ where $0<|\la|<1, F(\la\sin{z})\neq\emptyset,$ \cite[Theorem 1]{pd}. On similar lines to above example, it can be seen that\,$\emptyset\neq F(\la\sin z)=F(\la\sin z+2\pi)=F(G),$ and hence $F(G)\neq\emptyset.$
\end{example}
The following definitions are well known in  transcendental semigroup theory:
\begin{definition}\label{sec5,defn2}
Let $G$ be a transcendental semigroup. A set $W$ is forward invariant under $G$ if $g(W)\subset W$ for all $g\in G$ and $W$ is backward invariant under $G$ if $g^{-1}(W)=\{w\in\C:g(w)\in W\}\subset W$ for all $g\in G.$ $W$ is called completely invariant under $G$ if it is both forward and backward invariant under $G.$ 
\end{definition}
It is easily seen for a transcendental semigroup $G,$ $F(G)$ is forward invariant and $J(G)$ is backward invariant, see \cite[Theorem 2.1]{poon1}.


\begin{definition}\label{sec5,defn3}
A component $U$ of $F(G)$ is called a wandering domain of $G$ if the set $\{U_g:g\in G\}$ is infinite (where as usual $U_g$ is the component of $F(G)$ containing $g(U)$). Otherwise $U$ is called a pre-periodic component of $F(G).$
\end{definition}
We now propose a natural definition of boundedness of a semigroup on a set which will be needed later on:
\begin{definition}\label{sec5,defn4}
A transcendental semigroup $G=[f_1,f_2,\ldots]$ is said to be bounded on a set $A$ if all the generators in $G$ are bounded on $A$.
\end{definition}
Recall that $w\in\C$ is a critical value of a transcendental entire function $f$ if there exist some $w_0\in\C$ with $f(w_0)=w$ and $f'(w_0)=0.$ Here $w_0$ is called a critical point of $f.$ The image of a critical point of $f$ is  critical value of $f.$ Also recall that $\zeta\in\C$ is an asymptotic value of a transcendental entire function $f$ if there exist a curve $\Gamma$ tending to infinity such that $f(z)\to \zeta$ as $z\to\ity$ along $\Gamma.$
We now introduce definitions of  critical point, critical value and asymptotic value of a transcendental semigroup $G$:

\begin{definition}\label{sec5,defn5}
A point $z_0\in\C$ is called a critical point of $G$ if it is a critical point of some $g\in G.$  A point $w\in\C$ is called a critical value of $G$ if it is a critical value of some $g\in G.$
\end{definition}
\begin{definition}\label{sec5,defn6}
 A point $w\in\C$ is called an asymptotic value of $G$ if it is an asymptotic  value of some $g\in G.$
\end{definition}

\section{Theorems and their proofs}\label{sec3}
We now give the generalisation of the results from classical theory of dynamics associated with the iteration of a transcendental entire function mentioned in Section \ref{sec2}, to the semigroup of transcendental entire functions. For a multiply connected component $U$ of $F(G),$ recall  Definition \ref{sec5,defn1} of the subsemigroup $\ti\G_U$ of $G$.

\begin{theorem}\label{sec3,thm1}
Let $G$ be a transcendental semigroup and $U$  a multiply connected component of $F(G).$ Then for all $g\in\ti\G_U$
\begin{enumerate}
\item\ $g^n\to\ity$ locally uniformly on $U;$
\item\ Suppose $\ga$ is a Jordan curve that is not contractible in $U.$ Then distance between $g^n(\gamma)$ and $0$ is large for all sufficiently large $n.$  Also $ind_0\,g^n(\ga)>0$ for all sufficiently large $n$ and $ind_0\,g^n(\ga)\to\ity$ as $n\to\ity$, (where $ind_{\zeta}\,\ga$ denotes the index of a curve $\ga\subset\C$ with respect to a point $\zeta$).
\end{enumerate}
\end{theorem}

\begin{proof}
\begin{enumerate}
\item\ We have $F(G)\subset F(g)$ for all $g\in G$. For  $g\in\ti\G_U,$ let $\ti U_g $ be a multiply connected component of $F(g)$ containing $U.$ From Theorem \ref{sec2,thm3}, $g^n\to\ity$ locally uniformly on $\ti U_g$, and hence on $U$.
\item\ Let $\ga$ be a Jordan curve that is not contractible in $U.$  As $g^n\to\ity$ locally uniformly on $U$ and $\ga$ is a compact subset of $U,$ so the distance between $g^n(\gamma)$ and $0$ is sufficiently large as $n\to\ity.$ Assume there is a subsequence $\{n_k\}$ for which $ind_0\,g^{n_k}(\ga)=0.$ Then $g^{n_k}$ does not have zeros inside $\gamma,$ for if $g^{n_k}(z_0)=0$, for some $z_0$ inside $\ga,$ we  have $0=g^{n_k}(z_0)$ lies inside $g^{n_k}(\ga),$ a contradiction to $ind_0\,g^{n_k}(\ga)=0.$ By the minimum principle, $g^{n_k}\to\ity$ inside $\ga,$ which is not possible as inside $\ga$ there are points of $J(g),$ and hence $ind_0\,g^n(\ga)>0$ for $n$ sufficiently large. Also as distance of $g^n(\ga)$ from $0$ approaches $\ity$ as $n\to\ity,$ we have $ind_0\,g^n(\ga)\to\ity$ as $n\to\ity.$\qedhere
\end{enumerate}
\end{proof}
Recall Definition \ref{sec5,defn4} of boundedness of a semigroup on a set.
\begin{theorem}\label{sec3,thm2}
Let $G$ be a transcendental semigroup which is bounded on some curve $\Gamma$ going to $\ity.$ Then all  components of $F(G)$ are simply connected.
\end{theorem}

\begin{proof}
For all $g\in G, g(\Gamma)$ is bounded and so from Theorem \ref{sec2,thm4}, all components of $F(g)$ are simply connected. If $U$ is a multiply connected component of $F(G)$ then for all $g\in\ti\G_U, F(g)$ has a multiply connected component $\ti U_g\supset U$ which is a contradiction and hence the result.
\end{proof}
\begin{theorem}\label{sec3,thma}
Let $G$ be a transcendental semigroup.Then the following holds:
\begin{enumerate}
\item [(i)] Every unbounded component  of $F(G)$ is simply connected;
\item [(ii)] If $F(G)$ has an unbounded component $V,$ then all  components of $F(G)$ are simply connected;
\item [(iii)]  Any doubly connected component $U$ of $F(G)$ does not contain critical points of $\ti\G_U;$
\item [(iv)]  If $U\subset F(G)$ is a domain which is completely invariant under some $g_0\in G,$ then  $U$ is unbounded and  all components of $F(G)$ are simply connected.
\end{enumerate}
\end{theorem}
\begin{proof}
\begin{enumerate}
\item [(i)] If $F(G)$ has an unbounded multiply connected component $U,$ then for all $g\in\ti\G_U, F(g)$ will have a multiply connected component $\ti{U_g}\supset U.$ But then $U$ is contained in an unbounded component of $F(g)$ for all $g\in\ti\G_U,$ and so by Theorem \ref{sec2,thm5}, all components of $F(g)$ are simply connected for all $g\in\ti\G_U$. Thus we arrive at a contradiction and hence the result.
\item [(ii)] As $V\subset F(g)$ for all $g\in G$, there exist for each $g\in G$ an unbounded component $V_g'$ of $F(g)$ with $V\subset V_g'.$ By Theorem \ref{sec2,thm6}, all components of $F(g)$ are simply connected for all $g\in G.$ It is now evident if $F(G)$ has a multiply connected component then we arrive at a contradiction and hence the result.
\item [(iii)] As $U$ is a doubly connected component of $F(G),$ so $U$ is contained in a doubly connected component $\ti U_g$ of $F(g)$ for $g\in\ti\G_U$. From Theorem \ref{sec2,thm9}, $\ti U_g$ does not contain critical points of $g$ for each $g\in\ti\G_U$ and so does $U$. 
\item [(iv)] $U\subset F(G)\subset F(g_0)$, $U$ is completely invariant under $g_0$ and so is unbounded by Theorem \ref{sec2,thmf}. 
 There exist a component say $U'$ of $F(G)$ which contains the unbounded domain $U$ and so is itself unbounded. By Theorem \ref{sec3,thma}(ii), all components of $F(G)$ are simply connected.\qedhere
\end{enumerate}
\end{proof}

Theorem \ref{sec2,thm10} has been generalised to the semigroup case by Zhigang \cite{zhigang}. Here we give another proof of Zhigang's result. We will need the following lemma:
\begin{lemma}\cite[Theorem 4.2]{poon1}\label{sec3,lem1}
If $G$ is a transcendental semigroup, then $J(G)=\overline{\bigcup_{g\in G} J(g)}.$
\end{lemma}
\begin{theorem}\label{sec3,thm11}
Let $G$ be a transcendental semigroup. Then a multiply connected component of $F(G)$ is bounded and wandering, and hence a pre-periodic component of $F(G)$ must be simply connected.
\end{theorem}
\begin{proof}
Suppose $V\subset F(G)$ is a multiply connected component. Then from Theorem \ref{sec3,thma}(i), $V$ is bounded. Let $\gamma\subset V$ be a curve which is not contractible in $V$ and whose interior contains points of $J(G).$ We now show the existence of a $g\in G$ such that $J(g)$ intersects the bounded interior portion  of $\gamma$. Denote the bounded interior portion of $\gamma$ by $\gamma_1.$ Let $\zeta_0\in J(G)\cap\gamma_1.$ Then from Lemma \ref{sec3,lem1}, there is a sequence $\{g_j\}$ in  $G$ such that there are points $\zeta_j\in J(g_j)$ with  $\zeta_j\to\zeta_0.$ We now pick a $g_{j_0}$ from this sequence $\{g_j\}$. As $\gamma\subset V\subset F(g_{j_{0}})$ and $\gamma_1\cap J(g_{j_0})\neq \emptyset,$ $\gamma$ is not contractible with respect to $F(g_{j_0})$ and so the component $\ti{U_{g_{j_0}}}$ of $F(g_{j_0})$ which contains $V$ is multiply connected. From Theorem \ref{sec2,thm10}, $\{\ti{U_{g_{j_0}}},{g_{j_0}}^n: n\in\N\}$ is infinite, and thus $V$ is a wandering domain of $G.$
\end{proof}
As a consequence of above theorem we have the following corollary:
\begin{corollary}\label{sec3,cora}
If $G$ is a transcendental semigroup, then $F(G)$ does not contain Herman rings.
\end{corollary}
\section{Finitely generated transcendental semigroups}\label{sec4}

We now consider finitely generated transcendental semigroups. A semigroup $G=[g_1,\ldots,g_n]$ generated by finitely many functions is called finitely generated. Furthermore if $g_i$ and $g_j$ are permutable, that is $g_i\circ g_j=g_j\circ g_i,$ for all $1\leq i,j\leq n$, then $G$ is called finitely generated abelian transcendental semigroup. Poon \cite{poon1}, investigated some properties of abelian transcendental semigroups and wandering domains of transcendental semigroups. Recall the Eremenko-Lyubich class
\[\mathcal{B}=\{f:\C\to\C\,\,\text{transcendental entire}: \text{Sing}{(f^{-1})}\,\text{is bounded}\},\]
(where Sing($f^{-1}$) is the set of critical values and asymptotic values of $f$ and their finite limit points). Each $f\in\mathcal{B}$ is said to be of bounded type. A transcendental entire function $f$ is of finite type if Sing $f^{-1}$ is a finite set. Furthermore if the transcendental entire functions $f$ and $g$ are of bounded type then so is $f\circ g$ as Sing $((f\circ g)^{-1})\subset$ Sing $f^{-1}\cup f(\text{Sing}(g^{-1})),$ see \cite{berg4}. 
We now state a result of Baker which we generalise to the semigroup case:
\begin{theorem}\cite[Lemma 4.4]{baker2}\label{sec4,thmab}
If $f$ and $g$ are transcendental entire functions and if $\ity$ is not a limit function of any subsequence of ($f^n$) in a component of $F(f),$ nor of a subsequence of ($g^n$) in a component of $F(g),$ then $F(f)=F(g).$
\end{theorem}
Before we generalise this result, we prove a lemma:
\begin{lemma}\label{sec4,lemmabc}
Let $f$ and $g$ be two  transcendental  entire functions satisfying $f\circ g=g\circ f.$ Then $F(f\circ g)\subset F(f)\cap F(g).$
\end{lemma}
\begin{proof}
In \cite{berg4}, it was shown that $z\in F(f\circ g)$ if and only if $f(z)\in F(g\circ f).$ Since $f\circ g=g\circ f,\, F(f\circ g)$ is completely invariant under $f$ and by symmetry, under $g$ respectively and so in particular it is forward invariant under them. So $f(F(f\circ g))\subset F(f\circ g)$ and $g(F(f\circ g))\subset F(f\circ g),$  which by Montel's Normality Criterion implies $F(f\circ g)\subset F(f)$ and $F(f\circ g)\subset F(g)$ and hence the result. 
\end{proof}

\begin{theorem}\label{sec4,thmbc}
Let $G=[g_1,\ldots, g_n]$  be a finitely generated abelian transcendental semigroup. If $\ity$ is not a limit function of any subsequence in $G$ in a component of $F(G),$ then $F(G)=F(g)$  for all $g\in G.$
\end{theorem}
\begin{proof}
From Theorem \ref{sec4,thmab}, $F(g_i)=F(g_j)$ for $1\leq i, j\leq n.$ Using permutability of each $g_i,$ any $g\in G$ can be represented as $g=g_1^{l_1}\circ\cdots\circ g_n^{l_n}.$ Also $g$ permutes with each $g_i$. Using Lemma \ref{sec4,lemmabc}, $\ity$ is not a limit function of any subsequence of ($g^n$) in a component of $F(g)$, and hence on applying Theorem \ref{sec4,thmab} again we get  $F(g)=F(g_i)$ for each $i.$ Hence $F(G)=F(g)$ for all $g\in G.$
\end{proof}
We now state a result of Poon concerning finitely generated abelian transcendental semigroup in which the generators are of finite type:
\begin{theorem}\cite[Theorem 5.1]{poon1}\label{sec4,thm1}
If $G=[g_1,\ldots, g_n]$  is a finitely generated abelian transcendental semigroup in which each $g_i,\, i=1,\ldots,n$ is of finite type, then $F(G)=F(g)$ for all $g\in G.$
\end{theorem}
\begin{remark}\label{sec4,rema}
Even if $G$ is a non abelian transcendental semigroup, we can have $F(G)=F(g)$ for all $g\in G$. This  has been shown in Example \ref{sec5,eg1}.  
\end{remark}

\begin{remark}\label{sec4,rem1}
Theorem \ref{sec4,thm1} can be generalised to a finitely generated abelian transcendental semigroup $G$ in which each of the generators are of bounded type.
\end{remark}
We will need the following  lemma:
\begin{lemma}\cite{berg4}\label{sec4,lema}
 If  $f$ and $g$ are transcendental entire functions of bounded type then so is $f\circ g$.
\end{lemma}
We next prove another lemma:
\begin{lemma}\label{sec4,lemcd}
 If $f$ and $g$ are transcendental entire functions of bounded type with $f\circ g=g\circ f,$ then $F(f)=F(g)$.
\end{lemma}
\begin{proof}
As $f$ and $g$ are of bounded type, the forward orbits of points in $F(f)$ and $F(g)$ do not tend to $\ity$ under $f$ and $g$ respectively, see \cite[Theorem 1]{el2}. Combining this with Theorem \ref{sec4,thmab}, one gets $F(f)=F(g).$
\end{proof}
\begin{theorem}\label{sec4,thm2}
For a finitely generated abelian transcendental semigroup $G=[g_1,\ldots,g_n]$ in which each $g_i,\, i=1,\ldots,n$ is of bounded type, $F(G)=F(g)$ for all $g\in G.$
\end{theorem}
\begin{proof}
For each $i, j,\, 1\leq i, j\leq n,\, g_i$ and $g_j$ are permutable. From Lemma  \ref{sec4,lemcd}, $F(g_i)=F(g_j).$ Also using the permutability of each $g_i,$ any $g\in G$ can be represented as $g=g_1^{l_1}\circ\cdots\circ g_n^{l_n}$ and using Lemma  \ref{sec4,lema} repeatedly one gets $g$ is of bounded type. Also $g$ permutes with each $g_i$ and hence on applying Lemma \ref{sec4,lemcd} again, we get $F(g)=F(g_i), 1\leq i\leq n.$ Hence we conclude that $F(G)=F(g)$ for all $g\in G.$
\end{proof}
It is well known see \cite[Proposition 3]{el2}, if $f\in\mathcal{B}$ then all the components of $F(f)$  are simply connected. This follows from Theorem \ref{sec2,thm3} and the fact that if $f$ is a transcendental entire function of class $\mathcal{B}$, then the forward orbit of points in $F(f)$ does not approach $\ity$, see \cite[Theorem 1]{el2}. This result gets  generalised to semigroup case:
\begin{theorem}\label{sec4,thme}
 Let $G=[g_1,\ldots,g_n]$ be a finitely generated transcendental semigroup in which each generator is of bounded type. Then all components of $F(G)$ are simply connected.
\end{theorem}
\begin{proof}
 From Lemma \ref{sec4,lema}, each $g\in G$ is of bounded type. If $U\subset F(G)$ is a multiply connected component, then for each $g\in \ti\G_U,$ $U\subset \ti U_g$ where $\ti U_g$ is a multiply connected component of $F(g).$ From above observation, for all $g\in G,$ all components of $F(g)$ are simply connected and thus we arrive at a contradiction and hence the result.
\end{proof}
For two  permutable transcendental entire functions $f$ and $g$ with $F(f)=F(g)$, if $W$ is a component of $F(f),$ then will $W$ be a component of $F(g)$ of the same class. This was a problem posed by Baker  which has an affirmative answer when $f$ and $g$ are of finite type, see \cite{ren}.
\begin{theorem}\cite{ren}\label{sec4,thmb}
Let $f$ and $g$ be permutable transcendental entire functions of finite type.  If $W$ is a superattractive stable domain, an attractive stable domain, a parabolic stable domain or a Siegel disk of $f,$ then $W$ is also a superattractive stable domain, an attractive stable domain, a parabolic stable domain or a Siegel disk of $g$, respectively.
\end{theorem}
The given result can be generalised to a finitely generated abelian transcendental semigroup $G$ in which each of the generators are of finite type.

\begin{theorem}\label{sec4,thma}
Let $G=[g_1,\ldots,g_n]$ be an abelian transcendental semigroup in which each $g_i,\, i=1,\ldots,n$ is of finite type. If $U\subset F(G)$ is a superattractive stable domain, an attractive stable domain, a parabolic stable domain or a Siegel disk of $G$, then $U$ is also a  a superattractive stable domain, an attractive stable domain, a parabolic stable domain or a Siegel disk of g respectively, for all $g\in G$.
\end{theorem}
\begin{proof}
Observe that if $U\subset F(G)$ is a stable basin, then $G=G_U.$ From classification of periodic components of $F(G)$ in Section \ref{sec5}, $U$ is a stable basin of $F(g)$ of the same type for each $g\in G_U=G$ and hence the result.
\end{proof}
The following result provides a condition under which $F(f)$ of a transcendental entire function $f$ does not contain any asymptotic values of $f$:
\begin{theorem}\cite[p.\ 72]{Hua}\label{sec4,thmd}
If $f$ is a transcendental entire function whose stable domains are bounded, then $F(f)$ does not contain any asymptotic values of  $f.$
\end{theorem}

We now provide a condition under which $F(G)$ of a transcendental semigroup $G$  does not contain  any asymptotic values of $G.$
\begin{theorem}\label{sec4,thmc}
Let $G=[g_1,\ldots,g_n]$ be an abelian transcendental semigroup. If all stable domains of $F(g_i),\,1\leq i\leq n,$ are bounded, then $F(G)$ does not contain any asymptotic values of $G.$
\end{theorem}

\begin{proof}
From Lemma \ref{sec4,lemmabc}, all stable domains of $F(g)$ are bounded for all $g\in G.$ Let $z_0\in F(G)$ be an asymptotic value of $G.$ Then $z_0\in F(g)$ for all $g\in G$. From Definition \ref{sec5,defn6}, $z_0$ is an asymptotic value of some $g\in G$ and thus we arrive at a contradiction by Theorem \ref{sec4,thmd} and hence the result.
\end{proof}
We now study conjugate semigroups. Recall two entire functions $f$ and $g$ are conjugate if there exist a conformal map $\phi:\C\to\C$ with $\phi\circ f=g\circ\phi.$ By a conformal map $\phi:\C\to\C$ we mean an analytic and univalent map of the complex plane $\C.$ 
\begin{definition}\label{sec4,defn1}
Two finitely generated semigroups $G$ and $G'$  are said to be conjugate under a conformal map $\phi:\C\to\C$ if 
\begin{enumerate}
\item\ they have same number of generators,
\item\ corresponding generators are conjugate under $\phi.$
\end{enumerate}
If $G=[g_1,\ldots,g_n],$ we represent the conjugate semigroup $G'$ of $G$ by $G'=[\phi\circ g_1\circ\phi^{-1},\ldots,\phi\circ g_n\circ\phi^{-1}],$ where $\phi:\C\to\C$ is the conjugating map. Furthermore if $G$ is abelian and each of its generators $g_i, 1\leq i\leq n,$ is of finite type, then so is $G'.$
\end{definition}
If $f$ and $g$ are two rational functions which are conjugate under some Mobius transformation $\phi:\ti\C\to\ti\C$, then it is well known \cite[p.\ 50]{beardon}, $\phi(F(f))=F(g).$ We now show the dynamics of $G$ and $G'$ are similar for a finitely generated abelian transcendental semigroup in which each generator is of finite type. We will use Theorem \ref{sec4,thm1}.
\begin{theorem}\label{sec4,thm3}
Let $G=[g_1,\ldots, g_n]$ be a finitely generated abelian transcendental semigroup in which each $g_i, 1\leq i\leq n$ is of finite type and let $G'=[\phi\circ g_1\circ\phi^{-1},\ldots,\phi\circ g_n\circ\phi^{-1}]$ be the conjugate semigroup, where $\phi:\C\to\C$ is the conjugating map. Then $\phi(F(G))=F(G').$
\end{theorem}
\begin{proof}
Denote for each $i, 1\leq i\leq n,\,\phi\circ g_i\circ\phi^{-1}$ by $g_i'.$ From Theorem \ref{sec4,thm1}, $F(G)=F(g_i), 1\leq i\leq n.$ Also as $G'$ is abelian and each of its generators is of finite type, on applying Theorem \ref{sec4,thm1} again we have  $\phi(F(g_i))=F(g_i'), 1\leq i\leq n.$ Thus
\begin{equation}
\begin{split}
\notag
\phi(F(G))
&=\phi(F(g_i))\\
&=F(g_i')\\
&=F(G').\qedhere
\end{split}
\end{equation}
\end{proof}


\begin{remark}\label{sec4,rem3}
The above theorem can be generalised to a finitely generated abelian transcendental semigroup $G$ in which each of the generators are of bounded type.
\end{remark}
As a consequence of Theorems \ref{sec4,thm1} and \ref{sec4,thma} the following corollary is immediate:

\begin{theorem}\label{sec4,thm4}
Suppose $G=[g_1,\ldots, g_n]$ is a finitely generated abelian transcendental semigroup in which each $g_i, 1\leq i\leq n$ is of finite type. Then $G$ has no wandering domains.
\end{theorem}
To prove the next result we will need the following lemma:
\begin{lemma}\cite{el2}\label{sec4,lem2}
If $f$ is a transcendental entire function of bounded type, then $f$ has no Baker domains.
\end{lemma}
\begin{theorem}\label{sec4,thm5}
Suppose $G=[g_1,\ldots, g_n]$ is a finitely generated  transcendental semigroup in which each $g_i, 1\leq i\leq n$ is of bounded type. Then $G$ has no Baker domains.
\end{theorem}
\begin{proof}
Suppose $U\subset F(G)$ is a Baker domain. From classification of periodic components of $F(G)$  in Section \ref{sec5}, for all $g\in G_U,  U$ is a subdomain of a Baker domain contained in $F(g).$ As each $g\in G$ is of bounded type, by Lemma \ref{sec4,lem2}, $F(g)$ has no Baker domains and so we arrive at a contradiction and hence the result.
\end{proof}
We now provide a condition under which $F(G)$ of a finitely generated transcendental semigroup $G$ in which each of the generators are of bounded type, contains no wandering domains of $G$. For any set $E,$ by $E'$ we denote the set of finite limit points of $E$.  We will need the following results:
\begin{theorem}\cite[Theorem 1]{cheng}\label{sec4,thml}
Let $G$ be a finitely generated transcendental semigroup. If $U$ is a wandering domain of $G$, then any limit function of $G$ on $U$ is either $\ity,$ or lies in $J(G)\cap (\cup_{f\in G}\,Sing f^{-1})'.$
\end{theorem}
\begin{theorem}\cite[Theorem 3]{cheng}\label{sec4,thmm}
Let $G$ be a finitely generated transcendental semigroup in which each generator is of bounded type. Then for all $z\in F(G),$ there does not exist any sequence $\{g_n\}$ in $G$ for which $g_n(z)\to\ity$ as $n\to\ity.$
\end{theorem}
\begin{theorem}\label{sec4,thmn}
Suppose $G=[g_1,\ldots, g_n]$ is a finitely generated  transcendental semigroup in which each $g_i, 1\leq i\leq n$ is of bounded type. If $J(G)\cap (\cup_{f\in G}\,\text{Sing} f^{-1})'=\emptyset,$ then $G$ has no wandering domains.
\end{theorem}
\begin{proof}
If $D\subset F(G)$ is a wandering domain of $G,$ then from Theorem \ref{sec4,thml}, any limit function of $G$ on $D$ is either $\ity,$ or lies in $J(G)\cap (\cup_{f\in G}\,\text{Sing} f^{-1})'$. From Theorem \ref{sec4,thmm}, the limit functions cannot be $\ity,$ and hence all the limit functions are contained in $J(G)\cap (\cup_{f\in G}\,\text{Sing} f^{-1})'=\emptyset,$ which is a contradiction and hence the result.
\end{proof}
 
Acknowledgement. We are thankful to Prof. Kaushal Verma, IISc Bangalore for his valuable comments and suggestions. We also thank the reviewer for his/her valuable comments to improve the readability of the paper.

\end{document}